\documentclass[11pt, a4paper]{amsart}
\usepackage[english]{babel}
\usepackage{amssymb,amsmath,amsthm,mathrsfs,bm}
\usepackage{color}
\usepackage[margin=1in]{geometry}
\usepackage{stmaryrd}
\usepackage{graphicx}

\newtheorem{theorem}[subsection]{Theorem}

\newtheorem{definition}[subsection]{Definition}
\newtheorem{lemma}[subsection]{Lemma}
\newtheorem{remark}[subsection]{Remark}
\newtheorem{proposition}[subsection]{Proposition}
\newtheorem{corollary}[subsection]{Corollary}

\newtheorem*{claim*}{Claim}
\newtheorem*{theorem*}{Theorem}

\def\bal{\begin{aligned}}
\def\eal{\end{aligned}}
\def\be{\begin{equation}}
\def\ee{\end{equation}}
\def\bcs{\begin{cases}}
\def\ecs{\end{cases}}
\def\={\;=\;}
\def\+{\,+\,}
\def\-{\,-\,}

\def\Z{{\mathbb Z}}
\def\N{{\mathbb N}}

\def\R{{\mathbb R}}

\def\cartier{\mathscr{C}_p}

\def\v#1{{\bf #1}}

\def\is{\equiv}
\def\mod#1{({\rm mod}\ #1)}
\def\ceil#1{\lceil #1\rceil}

\def\boldell{\bm{\ell}}

\definecolor{lightgrey}{rgb}{0.8, 0.84, 0.8}

\title{$p$-Linear schemes for sequences modulo $p^r$}
\author{Frits Beukers}
\begin{document}
\maketitle
\begin{abstract}
Many interesting combinatorial sequences, such as Ap\'ery numbers and Franel numbers,
enjoy the so-called Lucas property modulo almost all primes $p$.
Modulo prime powers $p^r$ such sequences have a more complicated
behaviour which can be described by matrix versions of the Lucas property called
$p$-linear schemes. They are examples of finite $p$-automata. In this paper we construct
such $p$-linear schemes and give upper bounds for the number of states which, for fixed
$r$, do not depend on $p$.
\end{abstract}

\section{Introduction}
In this paper we study infinite sequences of integers that have special properties
when they are reduced modulo a prime $p$ or a prime power $p^r$. The most familiar
sequences are those sequences $a_k$ that have the so-called Lucas property modulo 
almost all primes $p$,
$$
a_{k_mp^m+\cdots+k_1p+k_0}\is a_{k_m}\cdots a_{k_1}a_{k_0}\mod{p}
$$
for all $k_i$ with $0\le k_i\le p-1$. Examples are the exponential sequences like $2^k$,
the central binomial coefficients $\binom{2k}{k}$, the Ap\'ery and Franel-type numbers
given by
$$
\sum_{m=0}^k\binom{k}{m}^2\binom{k+m}{m}^2,\quad \sum_{m=0}^k\binom{k}{m}^\ell
$$
and countless others. Our first result concerns constant term sequences of the form
$$
g_k=\text{ct}[g(\v x)^k]\quad k\ge0,
$$
where $g(\v x)\in\Z[x_1^{\pm1}\cdots x_n^{\pm 1}]$ is a Laurent
polynomial in the $n$ variables $x_1,\ldots,x_n$ and 'ct' is the operation of
taking the constant term. In this paper we denote the Newton polytope of
$g(\v x)$ by $\Delta$. By $\Delta^\circ$ we denote the topological interior of $\Delta$, i.e
$\Delta$ minus its proper faces.
In \cite[Thm 1.4]{HeStr22} we find the following theorem.

\begin{theorem}\label{main}
Suppose that $\Delta^\circ$ has $\v 0$ as unique lattice point. 
Then, for any prime $p$ the constant term sequence $g_k$ has the Lucas property. 
\end{theorem}

To my knowledge the first explicit occurrence of this theorem is in the case
$s=1$ of \cite[Thm 3.3]{SaStra09}, which is unfortunately unpublished.

In Section \ref{sec:constant-term} we present a proof,
which is basically the same as in \cite{HeStr22}. 
Here are some examples of interest. In \cite[Example 2.3]{HeStr22} we find

$$
\binom{2k}{k}=\text{ct}\left[\left(x+2+\frac{1}{x}\right)^k\right].
$$ 
In \cite[Example 2.4]{HeStr22} we find
$$
\sum_{k_1,\ldots,k_n\ge0\atop k_1+\cdots+k_n=k}\left(\frac{k!}{k_1!\cdots k_n!}\right)^2
=\text{ct}\left[(1+x_1+\cdots+x_n)^k\left(1+\frac{1}{x_1}+\cdots+\frac{1}{x_n}\right)^k\right].
$$
In \cite[Example 2.5]{HeStr22} we find
$$
\sum_{m=0}^k\binom{m+k}{m}^2\binom{k}{m}^2=
\text{ct}\left[\left(\frac{(x + y)(z + 1)(x + y + z)(y + z + 1)}{xyz}\right)^k\right].
$$
These are the famous Ap\'ery numbers that occur in the irrationality proof of $\zeta(3)$.
Finally we mention
$$
\sum_{m=0}^k\binom{k}{m}^{n+1}=\text{ct}\left[(1+x_1)^k(1+x_2)^k\cdots(1+x_n)^k
\left(1+\frac{1}{x_1\cdots x_n}\right)^k\right]
$$
which is not hard to show.

In general one cannot expect the Lucas property to hold modulo
$p^r$ for $r>1$, or when $\Delta^\circ$ contains another interior lattice point besides $\v 0$.
But one can introduce additional companion sequences, which result in matrix versions
of the Lucas property. 

\begin{definition}\label{def:linearscheme}
A $p$-linear scheme modulo $p^r$ for a set of sequences $\v a_k=(a_{1,k},\ldots,a_{s,k})^t$, $k=0,1,\ldots$
is a set of $s\times s$-matrices $M_\ell$ with entries in $\Z$ such that
\be\label{definition1}
\v a_{kp+\ell}\is M_\ell\v a_k\mod{p^r}
\ee
for every $\ell\in\{0,1,\ldots,p-1\}$ and $k\ge0$. 

We say that any constant linear combination of the sequences $a_{1,k},\ldots,a_{s,k}$ 
is a sequence that can be generated by a $p$-linear scheme modulo $p^r$. We call $s$
the number of states of this $p$-linear scheme. 
\end{definition}

As explained in \cite{Straub22}, $p$-linear schemes are special instances of finite $p$-automata.
A simple corollary of Definition \ref{def:linearscheme} is
$$
\v a_{k_mp^m+\cdots+k_1p+k_0}\is M_{k_0}M_{k_1}\cdots M_{k_{m}}\v a_0\mod{p^r},
$$
where $0\le k_i\le p-1$ for all $i=0,\ldots,m$. So all one needs in order to
compute $\v a_k\mod{p^r}$
for any $k$ are the matrices $M_\ell$ for $\ell=0,1,\ldots,p-1$, the vector 
$\v a_0$ and the base $p$ expansion of $k$. Furthermore, if $r=1$, $s=1$ and $a_0=1$
then Definition \ref{def:linearscheme} comes down to the Lucas property modulo $p$.

There is also a formulation of a $p$-linear scheme in terms of generating power series. 
Let $\v F(t)=\sum_{k\ge0}\v a_kt^k$, a vector of generating series,
and $M(t)=\sum_{\ell=0}^{p-1}M_\ell t^\ell$. Then \eqref{definition1}
is equivalent to 
\be\label{definition2}
\v F(t)\is M(t)\v F(t^p)\mod{p^r}.
\ee

As an example we display a generalization of Theorem \ref{main}.
Let $\Delta_\Z^\circ$ be the set of lattice points in $\Delta^\circ$. 
Define the $p$-th \emph{Hasse-Witt matrix} $H$ associated to $1-tg(\v x)$ as the square matrix
indexed by $\Delta_\Z^\circ$, whose $\v u,\v v$-entry is given by
$$
H_{\v u\v v}(t)=\text{coefficient of $\v x^{p\v u-\v v}$ in }(1-tg(\v x))^{p-1}. 
$$
Note that the entries of $H$ are polynomials of degree $\le p-1$. For any $\v u\in\Delta_\Z^\circ$ define
$$
F_{\v u}(t)=\sum_{k\ge0}ct[\v x^{\v u}g(\v x)^k]t^k.
$$
We prove in Section \ref{sec:constant-term} the following statement.

\begin{theorem}\label{main2}
For any $\v u\in\Delta_\Z^\circ$ we have
$$
F_{\v u}(t)=\sum_{\v v\in\Delta_\Z^\circ}H_{\v u\v v}(t)F_{\v v}(t^p)\mod{p}.
$$
\end{theorem}

A simple example for congruences modulo higher powers $p^r$ is the 2-state 
$p$-linear scheme modulo $p^2$ for the sequences $a=2^k,b_k=k2^k$.
It follows from the following congruences, which are easily verified for any odd prime $p$,
\begin{eqnarray*}
a_{kp+\ell}&\is& 2^\ell a_k+p\alpha 2^\ell b_k\mod{p^2}\\
b_{kp+\ell}&\is& \ell 2^\ell a_k+p2^\ell(1+\ell\alpha)b_k\mod{p^2},
\end{eqnarray*}
where $\alpha\is \frac{2^{p-1}-1}{p}\mod{p}$. In general $2^k\mod{p^r}$ can be generated by
an $r$-state $p$-linear scheme. 

An example of historical interest is Ira Gessel's mod $p^2$ congruence for the Ap\'ery numbers
$A_k=\sum_{m=0}^k\binom{k}{m}^2\binom{m+k}{m}^2$ in \cite{Ge82}.

\begin{theorem}[Gessel, 1982]
Define for all integers $k\ge0$,
$$A'_k=\sum_{m=0}^k\binom{k}{m}^2\binom{m+k}{m}^2
\left(\frac{1}{m+1}+\cdots+\frac{1}{k}\right).$$
Then, for any prime $p$ and any integers $k\ge0$ and $\ell=0,1,\ldots,p-1$ we have
$$
A_{kp+\ell}\is A_\ell A_k+pA'_\ell kA_k\mod{p^2}.
$$
\end{theorem}
Multiply this result with $kp+\ell$ to get
$$
(kp+\ell)A_{kp+\ell}\is \ell A_\ell A_k+p(A_\ell+\ell A'_\ell)kA_k\mod{p^2}.
$$
So the sequences $A_k,kA_k$ form the two states of a $p$-linear scheme modulo $p^2$.

Suppose again that $g(\v x)$ is any Laurent polynomial in $x_1,\ldots,x_n$ and coefficients 
in $\Z$. The following statement is an immediate consequence of Theorem \ref{lucaslemma}
and Remark \ref{states-bound}.

\begin{corollary}\label{main3}
Let $r\in\N$ and $p$ a prime. Choose an integer $\rho$ such that $\rho-\ceil{\rho/p}\ge r-1$. 
Then the numbers $g_k\mod{p^r}$ can be generated by a $p$-linear scheme with $s$ states, where
$s$ is bounded by $\rho$ times the number of lattice points in $\rho\Delta^\circ$.
\end{corollary}

Note that if $r=1$ we can take $\rho=1$. Since $g_0=1$ we recover Theorem \ref{main}
if $\v 0$ is the unique lattice point in $\Delta^\circ$.
Notice also that $\rho$ can always be chosen such that $\rho\le2r$. 

Our second result concerns coefficients of power series expansions for rational functions.
Let $P(\v x),Q(\v x)\in\Z[x_1,\ldots,x_n]$ and suppose that $P_0:=P(\v 0)\ne0$. Then we
can expand
$$
\frac{Q(\v x)}{P(\v x)}=\sum_{\v k\in\Z_{\ge0}^n}a_{\v k}\v x^{\v k}
$$
with $a_{\v k}\in\Z[1/P_0]$. Of particular interest in the literature are the diagonal
sequences $a_{k,k,\ldots,k}$. Some typical examples,
\begin{itemize}
\item[-] The diagonal elements of $(1-x_1-\cdots-x_n)^{-1}$ are $\frac{(kn)!}{(k!)^n}$.
\item[-] The diagonal elements of $((1-x_1-x_2)(1-x_3-x_4)-x_1x_2x_3x_4)^{-1}$ are
the Ap\'ery numbers introduced above. See \cite{Straub14}.
\end{itemize}
In \cite{We} we find a long list of further examples, with further references to
Sloane's Online Encyclopedia. 

In Theorem \ref{rationalcoefficients} we construct a $p$-linear scheme to compute the coefficients 
$a_{\v k}$. This time the coefficients $a_{\v k}$ have $\Z_{\ge0}^n$ as domain, rather than 
$\Z_{\ge0}$. However, it is straightforward to replace $\v a_k$ by
$\v a_{\v k}$ in Definition \ref{def:linearscheme}. Here is a consequence of
Corollary \ref{rationalcoefficients-corollary} which also follows from \cite[Theorem 5.2]{RoYa15}
with $s=1$.

\begin{corollary}
Let notations be as above and suppose that $Q(\v x)=1$, $P_0=1$ and the degree of $P(\v x)$ in
each variable is $1$. Then, for every prime $p$, 
$\v k\in\Z_{\ge0}^n$ and $\boldell\in[0,1,\ldots,p-1]^n$ we have
$$
a_{p\v k+\boldell}\is a_{\boldell} a_{\v k}\mod{p}.
$$
\end{corollary}

To prove this, apply Corollary \ref{rationalcoefficients-corollary} with $r=1$ (hence $\rho=1$)
and $d_1=\cdots=d_n=1$. Then we get a 1-state $p$-linear scheme to compute
the numbers $a_{\v k}\mod{p}$. Since $a_{\v 0}=1$ our corollary follows. Of course the
diagonal sequence $a_{k,k,\ldots,k}$ has now the Lucas property modulo every prime $p$. 

In \cite{RoYa15} and \cite{RoZe14} the authors propose algorithms to construct finite $p$-automata
and in particular $p$-linear schemes for the computation of sequences $a_k\mod{p^r}$.
Unfortunately the
algorithms seem rather inefficient (except in 'small' cases) and the resulting upper
bounds for the number of states
seem exceedingly large with a strong dependence on $p$ (see \cite[Thm 2.4]{Straub22}).
These papers were the motivation for writing the
present note. Its method is a simplified version of the techniques in \cite{BeVlIII},
which studies the action of the so-called Cartier operator on spaces of rational functions.
However, the present paper can be completely understood without having to read \cite{BeVlIII}.

{\bf Acknowledgement}: Many thanks to Armin Straub for the stimulating discussions I had on
the subject of this paper and his encouragement. 

Many thanks also to the referee for the careful reading of this paper and for the
remarks that very much improved its exposition.

\section{Main tool}\label{sec:main}
The starting point in this paper is a Laurent polynomial $f$ in the variables $x_1,\ldots,x_n$
and coefficients in a ring $R$ which is either $\Z$ or $\Z[t]$. Let $p$ be any prime and let
$\sigma$ be the Frobius lift on $R$, which is the identity map on $\Z$ and which maps $t$
to $t^p$ in the ring $\Z[t]$.

Let $\mu$ be a bounded polytope in $\R^n$ with a finite set of proper faces removed and such that
$\Delta^\circ\subset\mu$. The original polytope is denoted by $\overline\mu$, the closure
of $\mu$.
In \cite{BeVlI},\cite{BeVlII},\cite{BeVlIII} we study the action of the Cartier
operator $\cartier$ on the $R$-module $\Omega_f(\mu)$
of rational functions of the form $(r-1)!A(\v x)/f(\v x)^r$, where $A(\v x)$ is a Laurent
polynomial with coefficients in $R$ and support in $r\mu$. In Section \ref{sec:constant-term}
we use $\mu=\Delta^\circ$ and in Section \ref{powerseries} we use an extension of $\Delta^\circ$,
the so-called box closure. 
Note that in \cite{BeVlI} we use only $\mu$ such that $\overline\mu=\Delta$. 
The main property of these sets is the following.

\begin{lemma}\label{muproperty}
Let $\mu$ be a bounded polytope with a finite number of proper subfaces removed. Denote the original
polytope by $\overline\mu$. Then, for all positive integers $a,b$ we have
$$
a\mu+b\overline\mu\subset(a+b)\mu.
$$
Here we use the notation $X+Y=\{\v x+\v y|\v x\in X,\v y\in Y\}$ for any two sets $X,Y\subset\R^n$ and
$kX$ the $k$-fold sum of $k$ copies of $X$.  
\end{lemma}

Suppose we are given an infinite Laurent
series $\sum_{\v k}a_{\v k}\v x^{\v k}$, $a_{\v k}\in R$, with support in a positive cone
with $\v 0$ as top. Then $\cartier$ is defined by
$$
\cartier\left(\sum_{\v k}a_{\v k}\v x^{\v k}\right)=\sum_{\v k}a_{p\v k}\v x^{\v k}.
$$
The action of $\cartier$ on a rational function $R(\v x)\in\Omega_f(\mu)$ is obtained
by expansion of $R(\v x)$ in an infinite Laurent series and then application of 
$\cartier$. Unfortunately the image is a Laurent series which usually does not come from
an element in $\Omega_f(\mu)$.
Fortunately, in \cite{BeVlI} it is shown that the image does lie in the $p$-adic completion
of $\Omega_{f^\sigma}(\mu)$. Here $f^\sigma$ is simply $f$ with $\sigma$ applied to its
coefficients. 

In the present paper we will only need a mod $p^r$ version of this fact and for its
derivation we need the following easy lemma.

\begin{lemma}
Let $K(\v x),L(\v x)$ be two infinite Laurent series in $x_1,\ldots,x_n$ supported in
the same cone with top $\v 0$. Then,
$$
\cartier(K(\v x^p)L(\v x))=K(\v x)\cartier(L(\v x)).
$$
\end{lemma}

Let $r$ be a positive integer and consider the
module of rational functions $\frac{A(\v x)}{f(\v x)^r}$ with support of $A(\v x)$ in $r\mu$
by $M(r)$. By $M^\sigma(r)$ we denote module generated by the elements of $M(r)$ where we have applied $\sigma$ to all coefficients of the rational functions.

\begin{proposition}\label{mainprop}
Let $r$ be a positive integer and $p$ a prime. Choose an integer $\rho$ such that
$\rho-\ceil{\rho/p}\ge r-1$. Then the Cartier operator $\cartier$ modulo $p^r$
maps $M(\rho)$ to $M^\sigma(\rho)$. 
\end{proposition}

\begin{proof}
Rewrite $A(\v x)/f(\v x)^{\rho}$ as 
$A(\v x)f(\v x)^{p\ceil{\rho/p}-\rho}/f(\v x)^{p\ceil{\rho/p}}$.
Then note that
$f(\v x)^p=f^\sigma(\v x^p)-pG(\v x)$ for some Laurent polynomial $G$
with coefficients in $R$ and support in $p\Delta$. Then we use the $p$-adic expansion
\[
\frac{A(\v x)}{f(\v x)^{\rho}}=
\frac{A(\v x)f(\v x)^{p\ceil{\rho/p}-\rho}}
{(f^\sigma(\v x^p)-pG(\v x))^{\ceil{\rho/p}}}
=\sum_{m\ge0}p^m{\ceil{\rho/p}+m-1\choose m}
\frac{G(\v x)^m}{f^\sigma(\v x^p)^{m+\ceil{\rho/p}}}
A(\v x)f(\v x)^{p\ceil{\rho/p}-\rho}.
\]
Apply $\cartier$. We find that
\[
\cartier\left(\frac{A(\v x)}{f(\v x)^{\rho}}\right)=\sum_{m\ge0}p^m
{\ceil{\rho/p}+m-1\choose m}\frac{Q_m(\v x)}{f^\sigma(\v x)^{m+\ceil{\rho/p}}},
\]
where $Q_m(\v x)=\cartier(A(\v x)G(\v x)^mf(\v x)^{p\ceil{\rho/p}-\rho})$. The support of 
$A(\v x)G(\v x)^mf(\v x)^{p\ceil{\rho/p}-\rho}$ is contained in
$$
\rho\mu +mp\Delta+(p\ceil{\rho/p}-\rho)\Delta\subset \rho\mu +(mp+p\ceil{\rho/p}-\rho)\overline\mu.
$$
According to Lemma \ref{muproperty} this is contained in $(mp+p\ceil{\rho/p})\mu$,
hence the support of $Q_m$ lies in $(m+\ceil{\rho/p})\mu$. Its coefficients are still in $R$.
The terms with $m\ge r$ all vanish modulo $p^r$. Hence

\begin{equation}\label{cartier}
\cartier\left(\frac{A(\v x)}{f(\v x)^{\rho}}\right)\is\sum_{m=0}^{r-1}p^m
{\ceil{\rho/p}+m-1\choose m}\frac{Q_m(\v x)}{f^\sigma(\v x)^{m+\ceil{\rho/p}}}\mod{p^r}.
\end{equation}
Since $m+\ceil{\rho/p}\le r-1+\ceil{\rho/p}\le\rho$ we see that modulo $p^r$ the
right hand side is in $M^\sigma(\rho)$.
\end{proof}

\section{Constant terms of powers of Laurent polynomials}\label{sec:constant-term}
As a warm up to the proof of Theorem \ref{lucaslemma} we give a simple proof.
\begin{proof}[Proof of Theorem \ref{main}]
Let notations and assumptions be as in Theorem \ref{main}. Consider the equality
$$
\frac{1}{1-tg(\v x)}=\frac{\sum_{k=0}^{p-1}t^kg(\v x)^k}{1-t^pg(\v x)^p}\is 
\frac{\sum_{k=0}^{p-1}t^kg(\v x)^k}{1-t^pg(\v x^p)}\mod{p}.
$$
Expand on both sides in a power series in $t$ with Laurent polynomials in $x_1,\ldots,x_n$
as coefficients. Apply the Cartier operator on both sides. We get
$$
\cartier\left(\frac{1}{1-tg(\v x)}\right)\is \frac{1}{1-t^pg(\v x)}
\cartier\left(\sum_{k=0}^{p-1}t^kg(\v x)^k\right)\mod {p}.
$$
Take any monomial $\v x^{\v w}$ in $\sum_{k=0}^{p-1}t^kg(\v x)^k$. In order to survive $\cartier$
we must have $\v w=p\v v$ for some lattice point $\v v$. Note that $p\v v\in (p-1)\Delta$,
hence $\v v\in\Delta^\circ$. By our assumption on the lattice points in $\Delta^\circ$, $\v v=\v0$.
We obtain
$$
\cartier\left(\frac{1}{1-tg(\v x)}\right)\is \frac{1}{1-t^pg(\v x)}
\sum_{k=0}^{p-1}g_kt^k\mod{p}.
$$
Take the constant term in $x_1,\ldots,x_n$ on both sides. We get
$$
\sum_{\ell\ge0}g_\ell t^\ell\is\left(\sum_{k=0}^{p-1}g_kt^k\right)
\left(\sum_{m\ge0}g_mt^{pm}\right)\mod{p}.
$$
By comparison of the coefficients of $t^\ell=t^{pm+k}$ on both sides we get $g_{pm+k}\is g_kg_m\mod{p}$,
which proves our theorem.
\end{proof}

We now switch to the general situation. 
We apply Proposition \ref{mainprop} to construct $p$-linear schemes for sequences of
the form 
$$
\text{ct}[q(\v x)g(\v x)^k]\mod{p^r},
$$
where 'ct' is the operation of selecting the
constant term and $q(\v x),g(\v x)\in\Z[x_1^{\pm1},\ldots,x_n^{\pm1}]$ are Laurent polynomials. 
Let $f(\v x)=1-tg(\v x)$ and $\mu=\Delta^\circ$, the euclidean interior of
the Newton polytope $\Delta$.
Let $p$ be a prime, $r$ a positive integer and $\rho$ an
integer such that $r-1\le \rho-\ceil{\rho/p}$. For any monomial 
$\v x^{\v u}\in \rho\Delta^\circ$ expand
$\v x^{\v u}/f(\v x)^{\rho}$ as power series in $t$ and take the constant term of each
coefficient. We get a power series
$$
F_{\v u}(t):=\text{ct}\left[\frac{\v x^{\v u}}{(1-tg(\v x))^{\rho}}\right]=
\sum_{k\ge0}{\rho+k-1\choose \rho}\text{ct}[\v x^{\v u}g(\v x)^k]t^k,
$$
which is the generating series of the numbers ${\rho+k-1\choose \rho}
\text{ct}[\v x^{\v u}g(\v x)^k]$. 
The functions $t^\ell F_{\v u}(t)$ with $\ell=0,1,\ldots,\rho-1$ and $\v u\in\rho\Delta^\circ$
will be the states of our $p$-linear scheme. Usually we are interested in
sequences of the form $\text{ct}[g(\v x)^k]$, which correspond to the constant
term of $1/f(\v x)$. However, we can use the observation 
$$
\frac{1}{f(\v x)}=\frac{f(\v x)^{\rho-1}}{f(\v x)^{\rho}}.
$$
So the constant term of $1/f(\v x)\mod{p^r}$ is a $\Z$-linear combination of
the states $t^\ell F_{\v v}(t)$ with $\v v\in \rho\Delta^\circ$ and $\ell=0,1,\ldots,\rho-1$.

\begin{theorem}\label{lucaslemma}
Let $p,r,\rho$ and $F_{\v u}{t}$ be as before.
Then there exist polynomials $\lambda_{\v u,\ell,\v v,m}(t)$ in $t$ of degree $<p$ such that
$$
t^\ell F_{\v u}(t)\is\sum_{\v v\in \rho\Delta^\circ}
\sum_{m=0}^{\rho-1}\lambda_{\v u,\ell,\v v,m}(t)(t^p)^m
F_{\v v}(t^p)\mod{p^r}.
$$
for all $\v u\in (\rho\Delta^\circ)_\Z$ and $\ell=0,1,\ldots,\rho-1$. 
\end{theorem}

\begin{remark}\label{states-bound}
Note that we have found a $p$-linear scheme with states $t^\ell F_{\v u}\mod{p^r}$,
where $\ell=0,1,\ldots,\rho-1$ and $\v u\in \rho\Delta^\circ$. The cardinality of 
$(\rho\Delta^\circ)_\Z$ is bounded polynomially in $\rho$
of degree $n$. For every prime $p$ the choice $\rho=2r$ suffices so the number of such
states has order $O(r^{n+1})$, independent of $p$.
\end{remark}

\begin{proof}[Proof of Theorem \ref{lucaslemma}]
Consider \eqref{cartier} with $A(\v x)=\v x^{\v u}$ and expand as
\begin{equation}\label{lucas}
\cartier\left(\frac{\v x^{\v u}}{f(\v x)^{\rho}}\right)
\is \frac{1}{f^\sigma(\v x)^{\rho}}\sum_{\v v\in \rho\mu}q_{\v u,\v v}(t)\v x^{\v v}\mod{p^r}.
\end{equation}
The coefficients $q_{\v u,\v v}(t)$ are polynomials in $t$ and one easily verifies that
their degrees are at most $p(r-1)+p\ceil{\rho/p}-\rho$.
Let us choose polynomials $\lambda_{\v u,\ell,\v v,m}$
of degree $<p$ such that
$$
t^\ell q_{\v u,\v v}=\sum_{m=0}^{\rho-1}\lambda_{\v u,\ell,\v v,m}(t)t^{pm}.
$$
This is possible because the degree of the left hand side is at most
$$p(r-1)+p\ceil{\rho/p}-\rho+\rho-1=p(r-1+\ceil{\rho/p})-1\le p\rho-1.$$
The constant terms with respect to $x_1,\ldots,x_n$ on both sides of \eqref{lucas}
are equal modulo $p^r$ and by a straightforward computation our assertion follows.
\end{proof}

\begin{proof}[Proof of Theorem \ref{main2}]
Set $r=\rho=1$. Inspection of \eqref{cartier} with $A(\v x)=\v x^{\v u}$ yields
\be\label{cartier1}
\cartier\left(\frac{\v x^{\v u}}{1-tg(\v x)}\right)\is\frac{1}{1-t^pg(\v x)}
\cartier\left(\v x^{\v u}(1-tg(\v x))^{p-1}\right)\mod{p}.
\ee
The support of $\v x^{\v u}(1-tg(\v x))^{p-1}$ is in $p\Delta^\circ$ and we get
$$
\cartier\left(\v x^{\v u}(1-tg(\v x))^{p-1}\right)=
\sum_{\v v\in\Delta^\circ_\Z}H_{\v u\v v}(t)\v x^{\v v}.
$$
We complete the proof of the theorem by expansion of \eqref{cartier1} as power series
in $t$ and taking the constant term with respect to $x_1,\ldots,x_n$ on both sides.
\end{proof}

\section{Power series coefficients of rational functions}\label{powerseries}
Our second main application of Proposition \ref{mainprop} is to the coefficients
of power series expansions of rational functions.
Let $P$ be a polynomial in $x_1,\ldots,x_n$ with integer coefficients and a constant
term which is not divisible by the prime $p$. Let us call this constant $P_0$.
Let $\Delta$ be the Newton polytope of $P$.
For any polynomial $Q$ with integer coefficients we can expand $Q/P$ as
power series. Let us write
$$
\frac{Q(\v x)}{P(\v x)}=\sum_{\v k}a_{\v k}\v x^{\v k},
$$
and note that $a_{\v k}\in\Z[1/P_0]$ for all $\v k$. We will show that for every $r\ge1$
there is a $p$-linear scheme which generates the numbers $a_{\v k}\mod{p^r}$. 
In particular this holds for coefficients $a_{k\v v}$ with $\v v=(1,1,\ldots,1)$
and $k\in\Z_{\ge0}$. These are the diagonal elements of the power series. 

In Theorem \ref{rationalcoefficients} we use the concept of \emph{box closure}.
For any $\v x\in\R^n$ with non-negative entries we denote 
$B(\v x)=[0,x_1]\times\cdots\times[0,x_n]$ where $\v x=(x_1,\ldots,x_n)$.
We take $B(\v x)$ empty if $x_i<0$ for at least one index $i$.
For any subset $S\subset\R^n$ we define its box closure by $B(S)=\cup_{\v x\in S}B(\v x)$.

For us, the main property of box closures is that for any $\boldell\in\R_{\ge0}^n$ we
have $B(S-\boldell)\subset B(S)$, where $S-\boldell$ is the translate of $S$ over the vector $-\boldell$. 

In the following theorem we heavily use $B(\Delta^\circ)$, where $\Delta$ is the Newton polytope of $P$.

\begin{theorem}\label{rationalcoefficients}
Let $p$ be a prime not dividing $P_0$. Choose $\rho$ such that $r-1\le \rho-\ceil{\rho/p}$.
For every $\v u\in (\rho B(\Delta^\circ))_\Z$ consider the power series
$$
\frac{\v x^{\v u}}{P(\v x)^{\rho}}=\sum_{\v k}a_{\v u,\v k}\v x^{\v k},
$$
where $a_{\v u,\v k}\in\Z[1/P_0]$. Then there exist $\lambda_{\v u,\v v,\boldell}\in\Z$
with $\v u,\v v\in (\rho B(\Delta^\circ))_\Z$ and 
$\boldell\in[0,1,\ldots,p-1]^n$ such that for every $\v k$ 
we have
$$
a_{\v u,p\v k+\boldell}\is\sum_{\v v\in (\rho B(\Delta^\circ))_\Z}\lambda_{\v u,\v v,\boldell}
a_{\v v,\v k}\mod{p^r}.
$$
\end{theorem}

\begin{proof}
In this proof we shall use the convention that for any two sets $X,Y\subset\R^n$ we denote
$X+Y=\{\v x+\v y|\v x\in X,\v y\in Y\}$. Furthermore $kX$ with $k$ a positive integer denotes
the repeated sum of $k$ copies of $X$. 

We apply \eqref{cartier} in Proposition \ref{mainprop} with $f=P$ and $\mu=B(\Delta^\circ)$.
For any $\boldell\in[0,1,\ldots,p-1]^n$ and any $\v u\in \rho B(\Delta^\circ)$ we have
$$
\cartier\left(\frac{\v x^{\v u-\boldell}}{P(\v x)^{\rho}}\right)\is
\sum_{m=0}^{r-1}p^m\binom{\ceil{\rho/p}+m-1}{\rho}
\frac{Q_m(\v x)}{P(\v x)^{m+\ceil{\rho/p}}}\mod{p^r}.
$$
with 
$$
Q_m(\v x)=\cartier\left(\v x^{\v u-\boldell}G(\v x)^mP(\v x)^{p\ceil{\rho/p}-\rho}\right).
$$
In the present context we have that $G(\v x)$ is $P(\v x^p)-P(\v x)^p$
divided by $p$. The terms in 
$\v x^{\v u-\boldell}G(\v x)^mP(\v x)^{p\ceil{\rho/p}-\rho}$ have their support in 
\be\label{supportset}
\{\v u\}+\{-\boldell\}+pm\Delta+(p\ceil{\rho/p}-\rho)
\Delta\subset\{-\boldell\}+p(m+\ceil{\rho/p})B(\Delta^\circ).
\ee
Here we applied Lemma \ref{muproperty} with $\mu=B(\Delta^\circ),\v u\in\rho\mu$ and $\Delta\subset\overline\mu$. 
The operator $\cartier$ selects the terms whose support belongs to the set \eqref{supportset}
and which are divisible by $p$.
Since the components of $\boldell$ are all $\le p-1$ such support vectors have non-negative coordinates. 
By the main property of box closures the points in $\{-\boldell\}+p(m+\ceil{\rho/p})B(\Delta^\circ)$
with non-negative coordinates is contained in $p(m+\ceil{\rho/p})B(\Delta^\circ)$.

Hence we conclude that the support of $Q_m$ is in $(m+\ceil{\rho/p})B(\Delta^\circ)$. Note that by our choice of
$\rho$ we have $m+\ceil{\rho/p}\le r-1+\ceil{\rho/p}\le \rho$. Consequently there exist coefficients
$\lambda_{\v u,\v v,\boldell}\in\Z$ such that
$$
\cartier\left(\frac{\v x^{\v u-\boldell}}{P(\v x)^{\rho}}\right)\is
\sum_{\v v\in (\rho B(\Delta^\circ))_\Z}\lambda_{\v u,\v v,\boldell}
\frac{\v x^{\v v}}{P(\v x)^{\rho}}\mod{p^r}.
$$
The theorem follows by taking the power series expansion of both sides and then
selection of the coefficient of $\v x^{\v k}$ on both sides.
\end{proof}

\begin{corollary}\label{rationalcoefficients-corollary}
Let $a_{\v k}$ be as in the beginning of this section and suppose that the numerator $Q$ has its
support in $B(\Delta^\circ)$. Let $d_i$ be the degree of $P$
in the variable $x_i$ for $i=1,\ldots,n$. Let $r$ be a positive integer and $\rho$ an 
integer such that $r-1\le \rho-\ceil{\rho/p}$. 
Then there is a $p$-linear scheme, with at most $\rho^n d_1\cdots d_n$ states, which
generates the numbers $a_{\v k}\mod{p^r}$. 
\end{corollary}

\begin{proof}
We apply Theorem \ref{rationalcoefficients}. Note that the lattice points in $\rho B(\Delta^\circ)$
are contained in the box $[0,\rho d_1-1]\times\cdots\times[0,\rho d_n-1]$. 
Hence the number of lattice points is bounded above by $\rho^nd_1\cdots d_n$. 
One easily checks that $Q/P$ can be written as $\Z$-linear combination
of $\v x^{\v u}/P^{\rho}$ with $\v u\in (\rho B(\Delta^\circ))_\Z$. 
\end{proof}

\end{document}